\documentclass[12pt]{article}

\usepackage{amsmath,amssymb}
\usepackage{tikz}
\usepackage{color}
\usepackage[toc]{appendix}
\usepackage{graphicx}
\usepackage{fancyhdr}
\usepackage{enumitem}
\usepackage{bbm}
\usepackage{parskip}
\usepackage{float}
\usepackage{chngpage}
\usepackage{calc}
\usepackage{array}
\usepackage{booktabs}
\usepackage{rotating}
\usepackage{multirow}
\usepackage{adjustbox}
\usepackage{tabularx}
\usepackage{verbatim}
\usepackage{mathtools}
\usepackage{ragged2e}
\usepackage[makeroom]{cancel}
\usepackage{caption}
\usepackage{mathtools}
\usepackage{hyperref}
\usepackage{amsfonts,amsthm,bm}
\usepackage{appendix}
\usepackage{pgfplots}
\usepackage[english]{babel}
\usepackage{hyphenat}
\usepackage[makeindex]{imakeidx}
\usepackage{tikz}
\usetikzlibrary{matrix}
\usetikzlibrary{datavisualization.formats.functions}

\newtheorem{theorem}{Theorem}[section]
\newtheorem{proposition}[theorem]{Proposition}

\newtheorem{definition}[theorem]{Definition}

\newtheorem{remark}[theorem]{Remark}
\newtheorem{assumption}[theorem]{Assumption}

\makeatletter
\def\section{\@startsection {section}{1}{\z@}{3.25ex plus 1ex minus
		.2ex}{1.5ex plus .2ex}{\large\bf}}
\def\subsection{\@startsection{subsection}{2}{\z@}{3.25ex plus 1ex minus
		.2ex}{1.5ex plus .2ex}{\normalsize\bf}}
\@addtoreset{equation}{section} 
\makeatother

\title{Wong-Zakai approximations for quasilinear systems of It\^o's type stochastic differential equations}

\author{Alberto Lanconelli\thanks{Dipartimento di Scienze Statistiche Paolo Fortunati, Università di Bologna, Bologna, Italy. \textbf{e-mail}: alberto.lanconelli2@unibo.it} \and  Ramiro Scorolli\thanks{Dipartimento di Scienze Statistiche Paolo Fortunati, Università di Bologna, Bologna, Italy. \textbf{e-mail}: ramiro.scorolli2@unibo.it}}

\date{\today}

\begin{document}
	
\maketitle
	
\bigskip	

\begin{abstract}
	We extend to the multidimensional case a Wong-Zakai-type theorem proved by Hu and {\O}ksendal in \cite{HO} for scalar quasi-linear It\^o stochastic differential equations (SDEs). More precisely, with the aim of approximating the solution of a quasilinear system of It\^o's SDEs, we consider for any finite partition of the time interval $[0,T]$ a system of differential equations, where the multidimensional Brownian motion is replaced by its polygonal approximation and the product between diffusion coefficients and smoothed white noise is interpreted as a Wick product. We remark that in the one dimensional case this type of equations can be reduced, by means of a transformation related to the method of characteristics, to the study of a random ordinary differential equation. Here, instead, one is naturally lead to the investigation of a semilinear hyperbolic system of partial differential equations that we utilize for constructing a solution of the Wong-Zakai approximated systems. We show that the law of each element of the approximating sequence solves in the sense of distribution a Fokker-Planck equation and that the sequence converges to the solution of the It\^o equation, as the mesh of the partition tends to zero.      
\end{abstract}



\bigskip

\allowdisplaybreaks

\section{Introduction and statement of the main results}

Let $\{B(t)\}_{t\in [0,T]}$ be a standard one dimensional Brownian motion and, for a given finite partition $\pi$ of the interval $[0,T]$, denote by $\{{B}^{\pi}(t)\}_{t\in [0,T]}$ its polygonal approximation. Then, under suitable conditions on the coefficients $b:[0,T]\times\mathbb{R}\to\mathbb{R}$ and $\sigma:[0,T]\times\mathbb{R}\to\mathbb{R}$, the solution $\{Y^{\pi}(t)\}_{t\in [0,T]}$ of the random ordinary differential equation
\begin{align}\label{stra approx intro}
\frac{dY^{\pi}(t)}{dt}=b(t,Y^{\pi}(t))+\sigma(t,Y^{\pi}(t))\cdot\frac{d{B}^{\pi}(t)}{dt},
\end{align}
converges, as the mesh of $\pi$ tends to zero, to the strong solution $\{Y(t)\}_{t\in [0,T]}$ of the Stratonovich stochastic differential equation (SDE, for short)
\begin{align}\label{stra}
	dY(t)=b(t,Y(t))dt+\sigma(t,Y(t))\circ dB(t),
\end{align}
or equivalently (see \cite{KS}) of the It\^o SDE
\begin{align*}
dY(t)=\left[b(t,Y(t))+\frac{1}{2}\sigma(t,Y(t))\partial_y\sigma(t,Y(t))\right]dt+\sigma(t,Y(t))dB(t).
\end{align*}
This is the famous Wong-Zakai theorem \cite{WZ},\cite{WZ2} whose extension to the multidimensional case can be found in \cite{Stroock Varadhan}.\\     
In \cite{HO} the authors suggested how to modify equation (\ref{stra approx intro}) to get in the limit the It\^o's interpretation of (\ref{stra}): they considered the case with $\sigma(t,x)=\sigma(t)x$, where $\sigma:[0,T]\to\mathbb{R}$ is a deterministic function, and proved that the solution $\{X^{\pi}(t)\}_{t\in [0,T]}$ of the differential equation  
\begin{align}\label{ito approx intro}
\frac{dX^{\pi}(t)}{dt}=b(t,X^{\pi}(t))+\sigma(t)X^{\pi}(t)\diamond\frac{d{B}^{\pi}(t)}{dt},
\end{align}
converges, as the mesh of $\pi$ tends to zero, to the strong solution $\{X(t)\}_{t\in [0,T]}$ of the It\^o SDE
\begin{align}\label{ito 1D}
dX(t)=b(t,X(t))dt+\sigma(t)X(t)dB(t).
\end{align}
Here, the symbol $X^{\pi}(t)\diamond\frac{d{B}^{\pi}(t)}{dt}$ stands for the \emph{Wick product} between $X^{\pi}(t)$ and $\frac{d{B}^{\pi}(t)}{dt}$. (We postpone to the next section all the necessary mathematical details for the tools utilized here). Observe that the achievement of \cite{HO} is twofold: existence of a solution for (\ref{ito approx intro}) and its convergence towards the solution of (\ref{ito 1D}) (see also the related works \cite{BL} and \cite{DLS}). As far as the existence is concerned, equation (\ref{ito approx intro}) is not a standard random ordinary differential equation but instead an infinite dimensional partial differential equation. In fact, via equality
\begin{align}\label{Wick pointwise}
X^{\pi}(t)\diamond\frac{d{B}^{\pi}(t)}{dt}=X^{\pi}(t)\frac{d{B}^{\pi}(t)}{dt}-D_{\partial_tK^{\pi}(t,\cdot)}X^{\pi}(t),
\end{align}
where $K^{\pi}(t,\cdot)$ is a deterministic function that verifies the identity
\begin{align*}
B^{\pi}(t)=\int_0^TK^{\pi}(t,s)dB(s),
\end{align*}
while $D_{\partial_tK^{\pi}(t,\cdot)}$ stands for the directional Malliavin derivative along the function $s\mapsto \partial_tK^{\pi}(t,s)$, one recognizes equation (\ref{ito approx intro}) as a nonlinear evolution equation driven by an infinite dimensional gradient. Nevertheless, the particular form  of $\sigma(t,x)$ considered in \cite{HO} allows for a reduction method which transforms that into a random ordinary differential equation. We now briefly describe such method: we Wick-multiply both sides of (\ref{ito approx intro}) by
\begin{align*}
\mathtt{E}^{\pi}(0,t):=e^{-\int_0^t\sigma(s)\frac{d{B}^{\pi}(s)}{ds}ds-\frac{1}{2}\mathbb{E}\left[\left(\int_0^t\sigma(s)\frac{d{B}^{\pi}(s)}{ds}ds\right)^2\right]},\quad t\in [0,T],
\end{align*}  
to obtain
\begin{align*}
\frac{dX^{\pi}(t)}{dt}\diamond\mathtt{E}^{\pi}(0,t) =b(t,X^{\pi}(t))\diamond\mathtt{E}^{\pi}(0,t)+\sigma(t)X^{\pi}(t)\diamond\frac{d{B}^{\pi}(t)}{dt}\diamond\mathtt{E}^{\pi}(0,t),
\end{align*}
or equivalently,
\begin{align*}
\frac{dX^{\pi}(t)}{dt}\diamond\mathtt{E}^{\pi}(0,t) =b(t,X^{\pi}(t))\diamond\mathtt{E}^{\pi}(0,t)-X^{\pi}(t)\diamond\frac{d\mathtt{E}^{\pi}(0,t)}{dt}.
\end{align*}
Here, we utilized the identity
\begin{align*}
\frac{d\mathtt{E}^{\pi}(0,t)}{dt}=\sigma(t)\frac{d{B}^{\pi}(t)}{dt}\diamond\mathtt{E}^{\pi}(0,t).
\end{align*}
Rearranging the terms and exploiting the Leibniz rule for the Wick product we can write
\begin{align}\label{zx}
\frac{d}{dt}\left(X^{\pi}(t)\diamond\mathtt{E}^{\pi}(0,t)\right) =b(t,X^{\pi}(t))\diamond\mathtt{E}^{\pi}(0,t).
\end{align}
Now, if we set
\begin{align*}
\mathcal{X}^{\pi}(t):=X^{\pi}(t)\diamond\mathtt{E}^{\pi}(0,t),\quad t\in [0,T],
\end{align*}
and recall that
\begin{align*}
\mathtt{E}^{\pi}(0,t)\diamond\mathcal{E}^{\pi}(0,t)=1,\quad\mbox{ for all $t\in [0,T]$},
\end{align*}
where
\begin{align*}
\mathcal{E}^{\pi}(0,t):=e^{\int_0^t\sigma(s)\frac{d{B}^{\pi}(s)}{ds}ds-\frac{1}{2}\mathbb{E}\left[\left(\int_0^t\sigma(s)\frac{d{B}^{\pi}(s)}{ds}ds\right)^2\right]},\quad t\in [0,T],
\end{align*}  
we can reduce (\ref{zx}) to
\begin{align}\label{zx1}
\frac{d\mathcal{X}^{\pi}(t)}{dt} =b(t,\mathcal{X}^{\pi}(t)\diamond\mathcal{E}^{\pi}(0,t))\diamond\mathtt{E}^{\pi}(0,t).
\end{align}
Equation (\ref{zx1}) doesn't look simpler than (\ref{ito approx intro}); however, in (\ref{zx1}) one can apply the so-called Gjessing's Lemma which produces a Wick product-free expression. First, we observe that resorting to the definition of $\{{B}^{\pi}(t)\}_{t\in [0,T]}$ (see equation (\ref{polygonal}) below) one gets the representation
\begin{align*}
\int_0^t\sigma(s)\frac{d{B}^{\pi}(s)}{ds}ds=\int_0^T\sigma^{\pi}(t,s)dB(s),
\end{align*}
for a suitable $\sigma^{\pi}:[0,T]\times [0,T]\to \mathbb{R}$. With this notation at hand, Gjessing's formula can be simply stated as
\begin{align}\label{Gjessing 1}
\mathcal{Z}\diamond\mathcal{E}^{\pi}(0,t)=\mathtt{T}_{-\sigma^{\pi}(t,\cdot)}\mathcal{Z}\cdot\mathcal{E}^{\pi}(0,t),
\end{align}  
and
\begin{align}\label{Gjessing 2}
\mathcal{Z}\diamond\mathtt{E}^{\pi}(0,t)=\mathtt{T}_{\sigma^{\pi}(t,\cdot)}\mathcal{Z}\cdot\mathtt{E}^{\pi}(0,t),
\end{align}  
for a general random variable $\mathcal{Z}$ belonging to $\mathbb{L}^p(\Omega)$, for some $p> 1$. Here, $\mathtt{T}_{f}$ denotes the operator that translates the Brownian path by the function $\int_0^{\cdot}f(s)ds$. An application to equation (\ref{zx1}) of the last two identities leads to the random ordinary differential equation 
\begin{align}\label{zx2}
\frac{d\mathcal{X}^{\pi}(t)}{dt} =b\left(t,\mathcal{X}^{\pi}(t)\cdot\left(\mathtt{E}^{\pi}(0,t)\right)^{-1}\right)\cdot\mathtt{E}^{\pi}(0,t);
\end{align}
standard assumptions on the coefficients ensure  the existence of a unique solution $\{\mathcal{X}^{\pi}(t)\}_{t\in [0,T]}$ which, together with equality $X^{\pi}(t)=\mathcal{X}^{\pi}(t)\diamond\mathcal{E}^{\pi}(0,t)$, provides a unique solution also for (\ref{ito approx intro}). It is important to remark that the success of this reduction method is due to the opposite signs appearing in front of $\sigma^{\pi}(t,\cdot)$ in equations (\ref{Gjessing 1}) and (\ref{Gjessing 2}); this results in the disappearance of the translation operator, and hence of the Wick product, from equation (\ref{zx1}).

Aim of the present paper is the extension to the multidimensional case of the existence theorem for (\ref{ito approx intro}) and its convergence to (\ref{ito 1D}) proven in \cite{HO}. More precisely, for each finite partition $\pi$ of the interval $[0,T]$ we introduce the Cauchy problem
\begin{align}\label{WZ}
\begin{cases}
\frac{dX_i^{\pi}(t)}{dt}=b_i(t,X^{\pi}(t))+\sigma_i(t)X^{\pi}_i(t)\diamond \frac{dB_i^{\pi}(t)}{dt},\\
\quad\quad\mbox{for }t\in]0,T]\mbox{ and } i=1,...,d;\\
X^{\pi}_i(0)=c_i\in\mathbb{R},\quad\mbox{for } i=1,...,d,
\end{cases}
\end{align}
where $\{B^{\pi}(t)=(B^{\pi}_1(t),...,B^{\pi}_d(t))^*\}_{t\in[0,T]}$ stands for the polygonal approximation, relative to the partition $\pi$, of the standard $d$-dimensional Brownian motion $\{B(t)=(B_1(t),...,B_d(t))^*\}_{t\in [0,T]}$; the functions $b_1,...,b_d:[0,T]\times \mathbb{R}^d\to\mathbb{R}$ and $\sigma_1,...,\sigma_d:[0,T]\to\mathbb{R}$ are measurable while $c\in\mathbb{R}^d$ is a deterministic initial condition. System (\ref{WZ}) should be thought as a Wong-Zakai-type approximation for the system of It\^o's SDEs
\begin{align}\label{ItoSDE}
\begin{cases}
dX_i(t)=b_i(t,X(t))dt+\sigma_i(t)X_i(t)dB_i(t),\\
\quad\quad\mbox{for }t\in]0,T]\mbox{ and } i=1,...,d;\\
X_i(0)=c_i\in\mathbb{R},\quad\mbox{for } i=1,...,d.
\end{cases}
\end{align}
We will assume throughout the paper the following regularity properties for the coefficients: they guarantee the existence of a unique strong solution for (\ref{ItoSDE}).

\begin{assumption}\label{assumptions}\quad
	\begin{itemize}
		\item The functions $b(t,x)$, $\partial_{x_1}b(t,x)$,..., $\partial_{x_d}b(t,x)$ are bounded and continuous; 
		\item the functions $\sigma_1(t),...,\sigma_d(t)$ are bounded and continuous.
	\end{itemize}
\end{assumption}

Our first main theorem concerns the existence of a solution for (\ref{WZ}). It is worth mentioning that the reduction method described above doesn't apply to such systems, unless very particular cases are considered. In fact, the disappearance of the translation operator mentioned before takes place only when the same one dimensional Brownian motion drives all the equations in (\ref{WZ}) and moreover $\sigma_1(t)=\cdot\cdot\cdot=\sigma_d(t)$, for all $t\in [0,T]$. Therefore, to prove the existence of a solution for (\ref{WZ}) we have to employ a different approach which can be summarized as follows.\\
Using identity (\ref{Wick pointwise}) we rewrite (\ref{WZ}) as  
\begin{align}\label{WZ3}
\begin{cases}
\frac{dX_i^{\pi}(t)}{dt}=b_i(t,X^{\pi}(t))+\sigma_i(t)X^{\pi}_i(t)\frac{dB_i^{\pi}(t)}{dt}-\sigma_i(t)D^{(i)}_{K^{\pi}(t,\cdot)}X^{\pi}_i(t),\\
\quad\quad\mbox{for }t\in]0,T]\mbox{ and } i=1,...,d;\\
X^{\pi}_i(0)=c_i\in\mathbb{R},\quad\mbox{for } i=1,...,d.
\end{cases}
\end{align}
(Here, $D^{(i)}$ stands for the Mallivian derivative with respect to the $i$-th component of the multidimensional Brownian motion $\{B(t)\}_{t\geq 0}$). If we now divide the interval $[0,T]$ according to the partition $\pi=\{t_0,...,t_N\}$ and search on any subinterval $]t_k,t_{k+1}]$ for a solution to (\ref{WZ3}) of the form
\begin{align*}
X_i^{\pi}(t):=u_i(t,B^{\pi}(t_{k+1})-B^{\pi}(t_{k})),\quad t\in ]t_k,t_{k+1}], i=1,...,d
\end{align*} 
where $u_i:[0,T]\times\mathbb{R}^d\mapsto\mathbb{R}$ are deterministic functions, we see that $u=(u_1,...,u_d)$ has to solve a semilinear hyperbolic system of partial differential equations of the type
\begin{align}\label{PDE intro}
\begin{cases}
\partial_tu_i(t,x)=-\sigma_i(t)\partial_{x_i} u_i(t,x)+\sigma_i(t)\frac{x_i}{h}u_i(t,x)+b_i(t,u(t,x)),\\ 
\quad\quad\mbox{for }t\in]t_k,t_{k+1}], x\in\mathbb{R}^d\mbox{ and }i=1,...,d;\\
u_i(r,x)=\alpha_i,\quad\mbox{for } x\in\mathbb{R}^d\mbox{ and }i=1,...,d.
\end{cases}
\end{align}
Here, $h$ denotes the mesh of the partition $\pi$ while $\alpha_1,...,\alpha_d$ are suitable deterministic initial conditions. This link to the theory of partial differential equations allows us to state our first main result whose proof can be found in Section \ref{main theorem 1 proof}. We will deal with a weak notion of solution, see Definition \ref{def solution} below, that doesn't require any Malliavin differentiability property of the solution (as it should be implied by the last term in (\ref{WZ3})). 

\begin{theorem}[Existence]\label{main theorem 1}
	Under Assumption \ref{assumptions} equation (\ref{WZ}) possesses a mild solution $\{X^{\pi}(t)\}_{t\in [0,T]}$.
\end{theorem}

Our second main result shows that system (\ref{WZ}) is naturally connected with a Fokker-Planck-type equation which is solved in the sense of distributions by the law of the mild solution $\{X^{\pi}(t)\}_{t\in [0,T]}$. This establishes a further similarity between the Wong-Zakai approximating equation (\ref{WZ}) and its exact counterpart (\ref{ItoSDE}). This theorem generalizes the one obtained in \cite{Lanconelli FPWZ} for the scalar problem (\ref{ito approx intro}). The proof is postponed to Section \ref{main theorem 2 proof}.  

\begin{theorem}[Fokker-Planck equation]\label{main theorem 2}
	The law 
	\begin{align*}
	\mu^{\pi}(t,A):=\mathbb{P}(X^{\pi}(t)\in A),\quad t\in [0,T], A\in\mathcal{B}(\mathbb{R}^d)
	\end{align*}
	of the random vector $X^{\pi}(t)$ solves in the sense of distributions the Fokker-Planck equation 
	\begin{align}
		\left(\partial_t+\sum_{i,j=1}^d\sigma_i(t)x_ig_{ij}(t,x_i)\partial^2_{x_ix_j}+\sum_{i=1}^db_i(t,x)\partial_{x_i}\right)^*\mathtt{u}(t,x)=0,\quad t\in [0,T],x\in\mathbb{R}^d.
	\end{align}
Here, $g_{ij}:[0,T]\times\mathbb{R}\to\mathbb{R}$ is the measurable function defined in (\ref{g1}) and (\ref{g2}) below.	
\end{theorem}

Lastly, we present the convergence of $\{X^{\pi}(t)\}_{t\in [0,T]}$ towards the solution of the It\^o equation (\ref{ItoSDE}), as the mesh $\Vert\pi\Vert$ of the partition $\pi$ tends to zero. For the proof the reader is referred to Section \ref{main theorem 3 proof}.

\begin{theorem}[Convergence]\label{main theorem 3}
The mild solution $\{X^{\pi}(t)\}_{t\in [0,T]}$ converges, as the mesh of $\pi$ tends to zero, to the unique strong solution $\{X(t)\}_{\in [0,T]}$ of the It\^o SDE (\ref{ItoSDE}). More precisely,
\begin{align*}
\lim_{\Vert\pi\Vert\to 0}\sum_{i=1}^d\mathbb{E}\left[\left|X_i^{\pi}(s)-X_i(s)\right|\right]=0,\quad\mbox{ for all }t\in [0,T].
\end{align*}
\end{theorem}

The paper is organized as follows: in Section 2 we describe our framework and formalize all the mathematical concepts utilized in the introduction to present the problem; Section 3 contains the most novel part of our paper that consists in the link between the Wong-Zakai equation (\ref{WZ}) and the semilinear hyperbolic system of partial differential equations (\ref{PDE intro}); here, we describe in details the construction of the mild solution $\{X^{\pi}(t)\}_{t\in [0,T]}$; in Section 4 the proof of Theorem \ref{main theorem 2} on the Fokker-Planck equation passes through a careful interplay between the Gaussian nature of the noise and structure of the hyperbolic system; Section 5 concludes the manuscript with the proof of Theorem \ref{main theorem 3} which greatly benefits from the notion of mild solution introduced in Section 2.  

\section{Notation and preliminary results}\label{notation}

In this section we set the notation and prepare the ground for proving our main theorems. 
We fix a positive time horizon $T$ and a dimension $d\in\mathbb{N}$. Let $(\Omega,\mathcal{F},\mathbb{P})$ be the classical Wiener space over the time interval $[0,T]$ with values on $\mathbb{R}^d$ (see for instance \cite{Bogachev} or \cite{Nualart}); we denote by $\{B(t)=(B_1(t),...,B_d(t))^*\}_{t\in [0,T]}$ the coordinate process, i.e.
\begin{align*}
	B(t):\Omega&\to\mathbb{R}^d\\
	\omega&\mapsto B(t)(\omega):=\omega(t);
\end{align*}
by construction, $\{B(t)\}_{t\in [0,T]}$ is a standard $d$-dimensional Brownian motion.\\
We choose a finite partition $\pi:=\{t_0,...,t_N\}$ of the interval $[0,T]$, i.e.
\begin{align*}
0=t_0<t_1<\cdots<t_N=T,
\end{align*}
and set $\Vert\pi\Vert:=\max_{k\in\{0,1,...,N\}}|t_k-t_{k-1}|$. The real number $\Vert\pi\Vert$ is called \emph{mesh} of the partition $\pi$. We will assume without loss of generality that the partition is equally spaced, i.e. $t_k=\frac{kT}{N}$, for all $k\in\{0,...,N\}$; in this case we simply have $\Vert\pi\Vert:=\frac{T}{N}$ but we will continue to use the notation $\pi=\{t_0,...,t_N\}$ and $\Vert\pi\Vert$.\\
We associate to the partition $\pi$ the \emph{polygonal} approximation of the Brownian motion $\{B(t)\}_{t\in [0,T]}$:
\begin{align}\label{polygonal}
B^{\pi}(t):=\left(1-\frac{t-t_k}{t_{k+1}-t_k}\right)B(t_{k})+\frac{t-t_k}{t_{k+1}-t_k}B(t_{k+1}),\quad\mbox{if }t\in [t_{k},t_{k+1}[
\end{align}
and $B^\pi(T):=B(T)$. It is well known that for any $\varepsilon>0$ and $p\geq 1$ there exists a positive constant $C_{p,T,\varepsilon}$ such that
\begin{align*}
	\left(\mathbb{E}\left[\sup_{t\in [0,T]}|B(t)^{\pi}-B(t)|^p\right]\right)^{1/p}\leq C_{p,T,\varepsilon}\Vert\pi\Vert^{1/2-\varepsilon}.
\end{align*}
We refer the reader to Lemma 11.8 in Hu \cite{Hu} for a sharper estimate. For $i=1,...,d$, we set
\begin{align}\label{Sigma}
\Sigma_i(s,t):=\int_s^t\sigma_i(r)dr,\quad 0\leq s\leq t\leq T,
\end{align}
and observe that 
\begin{align}\label{q}
\int_s^t\sigma_i(r)\dot{B}_i^{\pi}(r)dr=&\int_s^{t_j}\sigma_i(r)\dot{B}_i^{\pi}(r)dr+\int_{t_j}^{t_{j+1}}\sigma_i(r)\dot{B}_i^{\pi}(r)dr+\cdot\cdot\cdot+\int_{t_k}^t\sigma_i(r)\dot{B}_i^{\pi}(r)dr\nonumber\\
=&\Sigma_i(s,t_j)\frac{B_i(t_j)-B_i(t_{j-1})}{h}+\Sigma_i(t_j,t_{j+1})\frac{B_i(t_{j+1})-B_i(t_{j})}{h}\nonumber\\
&+\cdot\cdot\cdot+\Sigma_i(t_k,t)\frac{B_i(t)-B_i(t_k)}{h},
\end{align}
when $t_{j-1}\leq s<t_j<\cdot\cdot\cdot<t_k\leq t$, for some $j\leq k$ in $\{1,...,N-1\}$. In particular, if $s,t\in [t_{k},t_{k+1}]$ for some $k\in\{0,...,N\}$, the last expression simplifies to
\begin{align*}
	\int_s^t\sigma_i(r)\dot{B}_i^{\pi}(r)dr=\Sigma_i(s,t)\frac{B_i(t_{k+1})-B_i(t_{k})}{h}.
\end{align*}
It is important to remark that according to (\ref{q}) the quantity $\int_s^t\sigma_i(r)\dot{B}_i^{\pi}(r)dr$ is a linear combination of independent Gaussian random variables with
\begin{align*}
\mathbb{E}\left[\int_s^t\sigma_i(r)\dot{B}_i^{\pi}(r)dr\right]=0
\end{align*} 
and 
\begin{align*}
	\mathbb{E}\left[\left(\int_s^t\sigma_i(r)\dot{B}_i^{\pi}(r)dr\right)^2\right]=\frac{1}{h}\left(\Sigma_i(s,t_j)^2+\Sigma_i(t_j,t_{j+1})^2+\cdot\cdot\cdot+\Sigma_i(t_k,t)^2\right)
\end{align*} 
We now set 
\begin{align*}
\mathcal{E}_i^{\pi}(s,t)&:=e^{\int_s^t\sigma_i(r)\dot{B}_i^{\pi}(r)dr-\frac{1}{2}\mathbb{E}\left[\left(\int_s^t\sigma_i(r)\dot{B}_i^{\pi}(r)dr\right)^2\right]}
\end{align*}
and observe that if $s,t\in [t_{k},t_{k+1}]$, for some $k\in\{0,...,N\}$, we get 
\begin{align*}
\mathcal{E}_i^{\pi}(s,t)=e^{\Sigma_i(s,t)\frac{B_i(t_{k+1})-B_i(t_{k})}{h}-\frac{1}{2h}\Sigma_i(s,t)^2}.
\end{align*}
It is easy to verify, using the independence of Brownian increments on disjoint subintervals $[t_k,t_{k+1}]$,
that
\begin{align}\label{flow property exponential}
\mathcal{E}_i^{\pi}(s,t_k)\mathcal{E}_i^{\pi}(t_k,t)=\mathcal{E}_i^{\pi}(s,t)
\end{align}
when $s\leq t_k\leq t$ for some $k\in\{1,...,N-1\}$ and $s,t\in [0,T]$.\\ 
A key role in the following will be played by the notion of \emph{Wick product}. The Wick product can be defined for any couple of random variables $X$ and $Y$ belonging to $\mathbb{L}^p(\Omega)$, for some $p> 1$ (see for instance \cite{Hu},\cite{HOUZ} or \cite{Janson}). Nevertheless, it is enough for our purposes to discuss the following two particular cases:
\begin{itemize}
\item if $X$ belongs to the Sobolev-Malliavin space $\mathbb{D}^{1,p}$, for some $p>1$ (see \cite{Nualart}), and $f\in L^2([0,T])$ is a deterministic function, then
\begin{align}\label{def wick first chaos}
X\diamond \int_0^Tf(t)dB_i(t):=X\cdot\int_0^Tf(t)dB_i(t)-D^{(i)}_{f}X,
\end{align}
with $D_f^{(i)}$ being the directional Mallivian derivative with respect to the $i$-th component of the multidimensional Brownian motion $\{B(t)\}_{t\geq 0}$ in the direction $f$;
\item  if $X\in\mathbb{L}^p(\Omega)$, for some $p>1$, and $s,t\in [t_{k},t_{k+1}]$, for some $k\in\{1,...,N-1\}$, we set
\begin{align}\label{def wick exponential}
X\diamond\mathcal{E}_i^{\pi}(s,t):=\mathtt{T}_{-\sigma_i,k}X\cdot \mathcal{E}_i^{\pi}(s,t),
\end{align}
where $\mathtt{T}_{-\sigma_i,k}$ stands for the \emph{translation} operator
\begin{align*}
(\mathtt{T}_{-\sigma_i,k}X)(\omega):=X\left(\omega-e_i\frac{\Sigma_i(s,t)}{h}\int_0^{\cdot}\boldsymbol{1}_{[t_k,t_{k+1}]}(r)dr\right).
\end{align*}
Here, $\{e_1,...,e_d\}$ denotes the canonical bases of $\mathbb{R}^d$ (recall that we are working with a $d$-dimensional Brownian motion and hence $\mathtt{T}_{-\sigma_i,k}$ acts only on the $i$-th component of $\{B(t)\}_{t\in [0,T]}$).
\end{itemize}
We observe that both definitions (\ref{def wick first chaos}) and (\ref{def wick exponential}) are actually consequences of the general definition of Wick product: the first one being related to the interplay between Wick product and Skorohod integral and the latter being nothing else than Gjessing's Lemma (recall the use we made of that in the introduction). Proofs of these facts as implications of the general definition of Wick product can be found in \cite{HOUZ}.  \\ 
It is known (\cite{Janson}) that the translation operator maps $\mathbb{L}^p(\Omega)$ into $\mathbb{L}^q(\Omega)$, for all $q<p$; therefore, since $\mathcal{E}_i^{\pi}(s,t)\in\mathbb{L}^p(\Omega)$ for any $p\geq 1$, we conclude that $X\diamond\mathcal{E}_i^{\pi}(s,t)$ belongs to $\mathbb{L}^q(\Omega)$, for all $q<p$. It is immediate to verify using definition (\ref{def wick exponential}) that
\begin{align*}
\mathcal{E}_i^{\pi}(s,t)\diamond\mathcal{E}_j^{\pi}(s,t)=\mathcal{E}_i^{\pi}(s,t)\cdot\mathcal{E}_j^{\pi}(s,t),\quad\mbox{ if $i\neq j$},
\end{align*}
and
\begin{align*}
	\mathcal{E}_i^{\pi}(s,t_k)\diamond\mathcal{E}_i^{\pi}(t_k,t)=\mathcal{E}_i^{\pi}(s,t_k)\cdot\mathcal{E}_i^{\pi}(t_k,t)=\mathcal{E}_i^{\pi}(s,t),\quad\mbox{ if $s\leq t_k\leq t\leq t_{k+1}$}.
\end{align*}
By means of the last identity we can extend definition (\ref{def wick exponential}) to the case where $s$ and $t$ do not necessarily belong to the same subinterval $[t_{k},t_{k+1}]$. In fact, assume that $t_{k-1}\leq s\leq t_{k}\leq t\leq t_{k+1}$: then,
\begin{align*}
X\diamond\mathcal{E}_i^{\pi}(s,t):=&(X\diamond\mathcal{E}_i^{\pi}(s,t_k))\diamond\mathcal{E}_i^{\pi}(t_k,t)\\
=&(\mathtt{T}_{-\sigma_i,k-1}X\cdot \mathcal{E}_i^{\pi}(s,t_k))\diamond\mathcal{E}_i^{\pi}(t_k,t)\\
=&\mathtt{T}_{-\sigma_i,k}(\mathtt{T}_{-\sigma_i,k-1}X\cdot \mathcal{E}_i^{\pi}(s,t_k))\cdot\mathcal{E}_i^{\pi}(t_k,t)\\
=&\mathtt{T}_{-\sigma_i,k}\mathtt{T}_{-\sigma_i,k-1}X\cdot\mathcal{E}_i^{\pi}(s,t_k)\cdot\mathcal{E}_i^{\pi}(t_k,t)\\
=&\mathtt{T}_{-\sigma_i,k}\mathtt{T}_{\sigma_i,k-1}X\cdot\mathcal{E}_i^{\pi}(s,t).
\end{align*}
The transformation (\ref{def wick exponential}) inherits from the translation operator a monotonicity property:
\begin{align*}
\mbox{if $X\leq Y$, then $X\diamond\mathcal{E}_i^{\pi}(s,t)\leq Y\diamond\mathcal{E}_i^{\pi}(s,t)$}.
\end{align*}
In particular,
\begin{align}\label{wick absolute value}
|X\diamond\mathcal{E}_i^{\pi}(s,t)|\leq |X|\diamond\mathcal{E}_i^{\pi}(s,t).
\end{align}

We are now able to formalize the solution concept that we utilize for solving (\ref{WZ}).

\begin{definition}\label{def solution}
A $d$-dimensional stochastic process $\{X^{\pi}(t)\}_{t\in [0,T]}$ is said to be a \emph{mild} solution of equation (\ref{WZ}) if:	
\begin{enumerate}
	\item the function $t\mapsto X^{\pi}(t)$ is almost surely continuous;
	\item for $i=1,...,d$ and $t\in [0,T]$, the random variable $X_i(t)$ belongs to $\mathbb L^p(\Omega)$ for some $p>1$;
	\item for $i=1,...,d$, the identity
	\begin{align}\label{solution}
	X_i^{\pi}(t)=c_i\mathcal E_i^{\pi}(0,t)+\int_0^t b_i(s,X^{\pi}(s))\diamond \mathcal E_i^{\pi}(s,t)ds,\quad t\in [0,T],
	\end{align}
	holds almost surely.
\end{enumerate}
\end{definition}
	
\begin{remark}
The way one can go from (\ref{WZ}) to (\ref{solution}) is pretty similar to the reduction method described in the introduction for the scalar case. Namely, if we Wick-multiply by $\mathtt E_i^{\pi}(0,t)$ both sides of  
\begin{align*}
\frac{dX_i^{\pi}(t)}{dt}=b_i(t,X^{\pi}(t))+\sigma_i(t)X^{\pi}_i(t)\diamond \frac{dB_i^{\pi}(t)}{dt},
\end{align*}
and employ the properties of Wick product mentioned there, we will end up with the corresponding multidimensional analogue of (\ref{zx1}), i.e.
\begin{align}\label{zx1 bis}
\frac{d\mathcal{X}_i^{\pi}(t)}{dt} =b_i(t,\mathcal{X}^{\pi}(t)\diamond\mathcal{E}_i^{\pi}(0,t))\diamond\mathtt{E}_i^{\pi}(0,t),
\end{align}
where
\begin{align*}
\mathcal{X}_i^{\pi}(t):=X_i^{\pi}(t)\diamond\mathtt{E}_i^{\pi}(0,t),\quad t\in [0,T].
\end{align*}
We now write (\ref{zx1 bis}) in the integral form
\begin{align*}
\mathcal{X}_i^{\pi}(t) =c_i+\int_0^tb_i(s,\mathcal{X}^{\pi}(s)\diamond\mathcal{E}_i^{\pi}(0,s))\diamond\mathtt{E}_i^{\pi}(0,s)ds;
\end{align*}
this identity together with
\begin{align*}
X_i^{\pi}(t)=\mathcal{X}_i^{\pi}(t)\diamond\mathcal{E}_i^{\pi}(0,t),
\end{align*}
gives
\begin{align*}
X_i^{\pi}(t)\diamond\mathtt{E}_i^{\pi}(0,t) =c_i+\int_0^tb_i(s,X_i^{\pi}(t))\diamond\mathtt{E}_i^{\pi}(0,s)ds.
\end{align*}
If we now Wick-multiply both sides above by $\mathcal{E}_i^{\pi}(0,t)$, we obtain (\ref{solution}). We recall that the application of Gjessing's Lemma here doesn't reduce the previous equation to a random ordinary differential equation and hence to prove the existence of a solution for (\ref{solution}) we have to resort to the technique described in the next section.
\end{remark}

\section{Proof of Theorem \ref{main theorem 1}}\label{main theorem 1 proof}

\subsection{An auxiliary semilinear hyperbolic system of PDEs}

To prove the existence of a mild solution for equation (\ref{WZ}) we introduce the following auxiliary semilinear hyperbolic system of partial differential equations
\begin{align}\label{PDE}
\begin{cases}
\partial_tu_i(t,x)=-\sigma_i(t)\partial_{x_i} u_i(t,x)+\sigma_i(t)\frac{x_i}{h}u_i(t,x)+b_i(t,u(t,x)),\\ 
\quad\quad\mbox{for }t\in]r,R], x\in\mathbb{R}^d\mbox{ and }i=1,...,d;\\
u_i(r,x)=\alpha_i,\quad\mbox{for } x\in\mathbb{R}^d\mbox{ and }i=1,...,d,
\end{cases}
\end{align}
where $\alpha_1,...,\alpha_d$ are constant initial conditions and $h$ denotes the mesh of the partition under consideration. The validity of Assumption \ref{assumptions} implies the existence of a unique classical solution for the Cauchy problem (\ref{PDE}) (see for instance \cite{Bressan} and \cite{Taylor}). \\
Now, if $u$ solves (\ref{PDE}), then from the trivial identity
\begin{align*}
\partial_{x_i}\left(u_i(t,x)e^{-\frac{|x|^2}{2h}}\right)=\partial_{x_i}u_i(t,x)e^{-\frac{|x|^2}{2h}}-\frac{x_i}{h}u_i(t,x)e^{-\frac{|x|^2}{2h}},
\end{align*}
we can argue that the function
\begin{align}\label{def v}
v(t,x):=u(t,x)e^{-\frac{|x|^2}{2h}},\quad t\in[r,R],x\in\mathbb{R}^d,
\end{align}  
is a classical solution of
\begin{align}\label{PDE v}
\begin{cases}
\partial_tv_i(t,x)=-\sigma_i(t)\partial_{x_i} v_i(t,x)+b_i\left(t,v(t,x)e^{\frac{|x|^2}{2h}}\right)e^{-\frac{|x|^2}{2h}},\\ 
\quad\quad\mbox{for }t\in]r,R], x\in\mathbb{R}^d\mbox{ and }i=1,...,d;\\
v_i(r,x)=\alpha_ie^{-\frac{|x|^2}{2h}},\quad\mbox{for } x\in\mathbb{R}^d\mbox{ and }i=1,...,d.
\end{cases}
\end{align}
Rewriting system (\ref{PDE v}) in the mild form
\begin{align*}
\begin{cases}
v_i(t,x)=\alpha_ie^{-\frac{|x-\Sigma_i(r,t)e_i|^2}{2h}}\\
\quad\quad\quad\quad
+\int_r^tb_i\left(t,v(s,x-\Sigma_i(s,t)e_i)e^{\frac{|x-\Sigma_i(s,t)e_i|^2}{2h}}\right)e^{-\frac{|x-\Sigma_i(s,t)e_i|^2}{2h}}ds\\
\mbox{for }t\in[r,R], x\in\mathbb{R}^d\mbox{ and }i=1,...,d,
\end{cases}
\end{align*}
(recall the definition of $\Sigma_i(s,t)$ in (\ref{Sigma})) and using identity (\ref{def v}), we obtain that $u$ solves
\begin{align*}
\begin{cases}
u_i(t,x)e^{-\frac{|x|^2}{2h}}=\alpha_ie^{-\frac{|x-\Sigma_i(r,t)e_i|^2}{2h}}\\
\quad\quad\quad\quad+\int_r^tb_i\left(t,u(s,x-\Sigma_i(s,t)e_i)\right)e^{-\frac{|x-\Sigma_i(s,t)e_i|^2}{2h}}ds\\
\mbox{for }t\in[r,R], x\in\mathbb{R}^d\mbox{ and }i=1,...,d,
\end{cases}
\end{align*}
or equivalently,
\begin{align}\label{PDE to SDE}
\begin{cases}
u_i(t,x)=\alpha_ie^{\frac{x_i}{h}\Sigma_i(r,t)-\frac{1}{2h}\Sigma_i(r,t)^2}\\
\quad\quad\quad\quad\quad
+\int_r^tb_i\left(t,u(s,x-\Sigma_i(s,t)e_i)\right)e^{\frac{x_i}{h}\Sigma_i(s,t)-\frac{1}{2h}\Sigma_i(s,t)^2}ds\\
\mbox{for }t\in[r,R], x\in\mathbb{R}^d\mbox{ and }i=1,...,d.
\end{cases}
\end{align}
Note that from the previous identity we get the estimate
\begin{align}\label{estimate}
|u_i(t,x)|\leq&|\alpha_i|e^{\frac{x_i}{h}\Sigma_i(r,t)-\frac{1}{2h}\Sigma_i(r,t)^2}\nonumber\\
&+\int_r^t|b_i\left(t,u(s,x-\Sigma_i(s,t)e_i)\right)|e^{\frac{x_i}{h}\Sigma_i(s,t)-\frac{1}{2h}\Sigma_i(s,t)^2}ds\nonumber\\
\leq&|\alpha_i|e^{\frac{x_i}{h}\Sigma_i(r,t)-\frac{1}{2h}\Sigma_i(r,t)^2}+M\int_r^te^{\frac{x_i}{h}\Sigma_i(s,t)-\frac{1}{2h}(\Sigma_i(s,t)^2}ds\nonumber\\
\leq&|\alpha_i|e^{\frac{x_i}{h}\Sigma_i(r,t)}+M\int_r^te^{\frac{x_i}{h}\Sigma_i(s,t)}ds.
\end{align}
Here, $M$ denotes a positive constant satisfying $|b_i(t,x)|\leq M$, for all $t\in [0,T]$, $x\in\mathbb{R}^d$ and $i=1,...,d$.

\subsection{Construction of a mild solution for (\ref{WZ})}

In the sequel, in order to stress the dependence on specific initial conditions, we will write 
\begin{align*}
u(t,x;r,\alpha)=(u_1(t,x;r,\alpha),...,u_d(t,x;r,\alpha))^*,\quad t\in [r,R], x\in\mathbb{R}^d
\end{align*} 
to denote the unique classical solution of (\ref{PDE}). We define the process $\{X^{\pi}(t)\}_{t\in [0,T]}$ inductively:
\begin{align}\label{def X_t}
X^{\pi}(t):=
\begin{cases}
u(t,B(t_1);0,c),&\mbox{ if }t\in [0,t_1];\\
u(t,B(t_2)-B(t_1);t_1,X^{\pi}(t_1)),&\mbox{ if }t\in ]t_1,t_2];\\
\quad\quad\cdot\cdot\cdot&\quad \cdot\cdot\cdot\\
u(t,B(T)-B(t_{N-1});t_{N-1},X^{\pi}(t_{N-1})),&\mbox{ if }t\in ]t_{N-1},T].
\end{cases}
\end{align}

We now verify that $X(t)$ is a mild solution of (\ref{WZ}), that is we check the conditions of Definition \ref{def solution}. 

The almost sure continuity of $t\in [0,T]\mapsto X(t)$ follows immediately from the continuity of $t\in[r,T]\mapsto u(t,x;r,\alpha)$, for all $x\in\mathbb{R}^d$ and $\alpha\in\mathbb{R}^d$ ($u$ is a classical solution of (\ref{PDE})) and the fact that for all $k\in \{1,...,N-1\}$ we have by construction
\begin{align*}
\lim_{t\to t_{k}^-}X^{\pi}(t)=\lim_{t\to t_{k}^+}X^{\pi}(t).
\end{align*}

We now verify that $X^{\pi}(t)\in\mathbb{L}^p(\Omega)$, for some $p>1$ and all $t\in [0,T]$. If $t\in [0,t_1]$, then by the definition of $X^{\pi}(t)$ and estimate (\ref{estimate}) we can write
\begin{align*}
|X^{\pi}_i(t)|=&|u(t,B(t_1);0,c)|\\
 \leq&|c_i|e^{\frac{B_i(t_1)}{h}\Sigma_i(0,t)}+M\int_0^te^{\frac{B_i(t_1)}{h}\Sigma_i(s,t)}ds,
\end{align*}
and hence
\begin{align}\label{estimate on first interval}
\Vert X^{\pi}_i(t)\Vert_p\leq&|c_i|\Vert e^{\frac{B_i(t_1)}{h}\Sigma_i(0,t)}\Vert_p+M\int_0^t\Vert e^{\frac{B_i(t_1)}{h}\Sigma_i(s,t)}\Vert_pds\nonumber\\
=&|c_i|e^{p\frac{\Sigma_i(0,t)^2}{2h}}+M\int_0^te^{p\frac{\Sigma_i(s,t)^2}{2h}}ds\nonumber\\
=&|c_i|e^{p\frac{\Sigma_i(0,t)^2}{2h}}+Mte^{\frac{p}{2h}\sup_{s\in [0,t]}\Sigma_i(s,t)^2}.
\end{align}
This proves the membership of $X^{\pi}_i(t)$ to $\mathbb{L}^p(\Omega)$, for all $i=1,...,d$, $t\in [0,t_1]$ and $p\geq 1$. Let us now take $t\in ]t_1,t_2]$; again, by the definition of $X^{\pi}(t)$ and estimate (\ref{estimate}) we can write
\begin{align*}
|X^{\pi}_i(t)|=&|u(t,B(t_2)-B(t_1);t_1,X^{\pi}(t_1))|\\
\leq&|X^{\pi}_i(t_1)|e^{\frac{B_i(t_2)-B_i(t_1)}{h}\Sigma_i(t_1,t)}+M\int_{t_1}^te^{\frac{B_i(t_2)-B_i(t_1)}{h}\Sigma_i(s,t)}ds,
\end{align*} 
and hence, using H\"older inequality,
\begin{align*}
\Vert X^{\pi}_i(t)\Vert_p\leq&\Vert |X^{\pi}_i(t_1)| e^{\frac{B_i(t_2)-B_i(t_1)}{h}\Sigma_i(t_1,t)}\Vert_p+M\int_{t_1}^t\Vert e^{\frac{B_i(t_2)-B_i(t_1)}{h}\Sigma_i(s,t)}\Vert_pds\\
\leq &\Vert X^{\pi}_i(t_1)\Vert_q\Vert e^{\frac{B_i(t_2)-B_i(t_1)}{h}\Sigma_i(t_1,t)}\Vert_{q'}+M\int_{t_1}^t\Vert e^{\frac{B_i(t_2)-B_i(t_1)}{h}\Sigma_i(s,t)}\Vert_pds\\
\leq &\Vert X^{\pi}_i(t_1)\Vert_q e^{q'\frac{\Sigma_i(t_1,t)^2}{2h}}+M\int_{t_1}^te^{p\frac{\Sigma_i(s,t)^2}{2h}}ds\\
\leq &\Vert X^{\pi}_i(t_1)\Vert_q e^{q'\frac{\Sigma_i(t_1,t)^2}{2h}}+M(t-t_1)e^{\frac{p}{2h}\sup_{s\in [t_1,t]}\Sigma_i(s,t)^2}.
\end{align*} 
This last estimate combined with (\ref{estimate on first interval}) provides the desired upper bound for $\Vert X^{\pi}_i(t)\Vert_p$ on the interval $]t_1,t_2]$. It also clear that in a similar manner one obtains analogous estimates for the $\mathbb{L}^p(\Omega)$-norm of $X_i(t)$ on any subinterval $]t_{k},t_{k+1}]$ for $k=2,...,N-1$.   \\

We are left with the verification that $\{X^{\pi}(t)\}_{t\in [0,T]}$ as defined in (\ref{def X_t}) satisfies identity (\ref{solution}). To this aim we prove the following auxiliary result.

\begin{proposition}
Identity (\ref{solution}) is equivalent to 
	\begin{align}\label{solution2}
		X_i^{\pi}(t)=X_i^{\pi}(t_{k-1})\diamond\mathcal E_i^{\pi}(t_{k-1},t)+\int_{t_{k-1}}^t b_i(s,X^{\pi}(s))\diamond \mathcal E_i^{\pi}(s,t)ds,\quad t\in [t_{k-1},t_{k}], 
	\end{align}
for all $k\in \{1,...,N\}$.
\end{proposition}

\begin{proof}
	Assume identity (\ref{solution}) to be true; then, for $t\in]t_{k-1},t_k]$ we can write
	\begin{align*}
		X_i^{\pi}(t)=&c_i\mathcal{E}_i^{\pi}(0,t)+\int_0^tb_i(s,X^{\pi}(s))\diamond\mathcal{E}_i^{\pi}(s,t)ds\\
		=&c_i\mathcal{E}_i^{\pi}(0,t)+\int_0^{t_{k-1}}b_i(s,X^{\pi}(s))\diamond\mathcal{E}_i^{\pi}(s,t)ds++\int_{t_{k-1}}^tb_i(s,X^{\pi}(s))\diamond\mathcal{E}_i^{\pi}(s,t)ds\\
		=&c_i\mathcal{E}_i^{\pi}(0,t_{k-1})\diamond\mathcal{E}_i^{\pi}(t_{k-1},t)+
		\int_{0}^{t_{k-1}} b_i(s,X^{\pi}(s))\diamond\mathcal{E}_i^{\pi}(s,t_{k-1})\diamond\mathcal{E}_i^{\pi}(t_{k-1},t)ds\\
		&+ \int_{t_{k-1}}^t b_i(s,X^{\pi}(s))\diamond\mathcal{E}_i^{\pi}(s,t)ds\\
		=&c_i\mathcal{E}_i^{\pi}(0,t_{k-1})\diamond\mathcal{E}_i^{\pi}(t_{k-1},t)+\left(
		\int_{0}^{t_{k-1}} b_i(s,X^{\pi}(s))\diamond\mathcal{E}_i^{\pi}(s,t_{k-1})ds\right)\diamond\mathcal{E}_i^{\pi}(t_{k-1},t)\\
		&+\int_{t_{k-1}}^t b_i(s,X^{\pi}(s))\diamond\mathcal{E}_i^{\pi}(s,t)ds\\
		=&\left(c_i\mathcal{E}_i^{\pi}(0,t_{k-1})+
		\int_{0}^{t_{k-1}} b_i(s,X^{\pi}(s))\diamond\mathcal{E}_i^{\pi}(s,t_{k-1})ds\right)\diamond\mathcal{E}_i^{\pi}(t_{k-1},t)\\
		&+\int_{t_{k-1}}^t b_i(s,X^{\pi}(s))\diamond\mathcal{E}_i^{\pi}(s,t)ds\\
		=&X_i^{\pi}(t_{k-1})\diamond\mathcal E_i^{\pi}(t_{k-1},t)+\int_{t_{k-1}}^t b_i(s,X^{\pi}(s))\diamond \mathcal E_i^{\pi}(s,t)ds.
	\end{align*}
This proves (\ref{solution2}). If we now start from (\ref{solution2}) and replace iteratively $X_i^{\pi}(t_{k-1})$ with
\begin{align*}
X_i^{\pi}(t_{k-2})\diamond\mathcal E_i^{\pi}(t_{k-2},t_{k-1})+\int_{t_{k-2}}^{t_{k-1}} b_i(s,X^{\pi}(s))\diamond \mathcal E_i^{\pi}(s,t)ds,
\end{align*}
and then replace $X_i^{\pi}(t_{k-2})$ with
\begin{align*}
X_i^{\pi}(t_{k-3})\diamond\mathcal E_i^{\pi}(t_{k-3},t_{k-2})+\int_{t_{k-3}}^{t_{k-2}} b_i(s,X^{\pi}(s))\diamond \mathcal E_i^{\pi}(s,t)ds,
\end{align*}
and so on, we will end up with (\ref{solution}).	
\end{proof}

\begin{remark}
We observe that according to the definition of $X^{\pi}(t)$ in (\ref{def X_t}), for any $k\in\{1,...,N\}$ and $t\leq t_{k-1}$ the random vector $X^{\pi}(t)$ depends only on the Brownian increments on the intervals $[0,t_1]$,...,$[t_{k-2},t_{k-1}]$. Therefore, the term
\begin{align*}
X_i^{\pi}(t_{k-1})\diamond\mathcal E_i^{\pi}(t_{k-1},t),\quad t\in ]t_{k-1},t_k]
\end{align*}
in (\ref{solution2}) can be rewritten for our particular mild solution as 
\begin{align*}
X_i^{\pi}(t_{k-1})\mathcal E_i^{\pi}(t_{k-1},t),\quad t\in ]t_{k-1},t_k].
\end{align*}
In fact, according to (\ref{def wick exponential}) one has
\begin{align*}
X_i^{\pi}(t_{k-1})\diamond\mathcal E_i^{\pi}(t_{k-1},t)&=\mathtt{T}_{-\sigma_i,k-1}X_i^{\pi}(t_{k-1})\mathcal E_i^{\pi}(t_{k-1},t)\\
&=X_i^{\pi}(t_{k-1})\mathcal E_i^{\pi}(t_{k-1},t).
\end{align*}
(The translation acts on a part of Brownian path which is disjoint from the increments on which $X_i^{\pi}(t_{k-1})$ depends).
\end{remark}

We are now ready to prove that $X^{\pi}(t)$ defined in (\ref{def X_t}) verifies identity (\ref{solution}) through the equivalent equalities (\ref{solution2}). Let $t\in [0,t_1]$; then, identity (\ref{PDE to SDE}) and definition (\ref{def X_t}) give
\begin{align*}
X^{\pi}_i(t)=&u_i(t,B(t_1);0,c)\\
=&c_ie^{\frac{B_i(t_1)}{h}\Sigma_i(0,t)-\frac{1}{2h}\Sigma_i(0,t)^2}\\
&+\int_0^tb_i\left(t,u(s,B(t_1)-\Sigma_i(s,t)e_i;0,c)\right)e^{\frac{B_i(t_1)}{h}\Sigma_i(s,t)-\frac{1}{2h}\Sigma_i(s,t)^2}ds\\
=&c_i\mathcal{E}_i^{\pi}(0,t)+\int_0^tb_i\left(t,u(s,B(t_1)-\Sigma_i(s,t)e_i;0,c)\right)\mathcal{E}_i^{\pi}(s,t)ds\\
=&c_i\mathcal{E}_i^{\pi}(0,t)+\int_0^t\mathtt{T}_{-\sigma_i,0}b_i\left(t,u(s,B(t_1);0,c)\right)\mathcal{E}_i^{\pi}(s,t)ds\\
=&c_i\mathcal{E}_i^{\pi}(0,t)+\int_0^tb_i\left(t,u(s,B(t_1);0,c)\right)\diamond\mathcal{E}_i^{\pi}(s,t)ds\\
=&c_i\mathcal{E}_i^{\pi}(0,t)+\int_0^tb_i\left(t,X^{\pi}(s)\right)\diamond\mathcal{E}_i^{\pi}(s,t)ds.
\end{align*}
This corresponds to (\ref{solution2}) for $t\in [0,t_1]$. Let us now consider the general subinterval $]t_{k},t_{k+1}]$, with $k\in \{1,...,N-1\}$; identity (\ref{PDE to SDE}) and definition (\ref{def X_t}) give
\begin{align*}
	X^{\pi}_i(t)=&u_i(t,B(t_{k+1})-B(t_k);t_k,X^{\pi}(t_k))\\
	=&X^{\pi}_i(t_k)e^{\frac{B_i(t_{k+1})-B_i(t_k)}{h}\Sigma_i(t_k,t)-\frac{1}{2h}\Sigma_i(t_k,t)^2}\\
	&+\int_{t_k}^tb_i\left(t,u(s,B(t_{k+1})-B(t_k)-\Sigma_i(s,t)e_i;t_k,X^{\pi}(t_k))\right)\times\\
	&\times e^{\frac{B(t_{k+1})-B(t_k)}{h}\Sigma_i(s,t)-\frac{1}{2h}\Sigma_i(s,t)^2}ds\\
	=&X^{\pi}_i(t_k)\mathcal{E}_i^{\pi}(t_k,t)\\
	&+\int_{t_k}^tb_i\left(t,u(s,B(t_{k+1})-B(t_k)-\Sigma_i(s,t)e_i;t_k,X^{\pi}(t_k))\right)\mathcal{E}_i^{\pi}(s,t)ds\\
	=&X_i^{\pi}(t_k)\mathcal{E}_i^{\pi}(t_k,t)+\int_{t_k}^t\mathtt{T}_{-\sigma_i,k}b_i\left(t,u(s,B(t_{k+1})-B(t_k);t_k,X^{\pi}(t_k))\right)\mathcal{E}_i^{\pi}(s,t)ds\\
	=&X_i^{\pi}(t_k)\mathcal{E}_i^{\pi}(t_k,t)+\int_{t_k}^tb_i\left(t,u(s,B(t_{k+1})-B(t_k);t_k,X^{\pi}(t_k))\right)\diamond\mathcal{E}_i^{\pi}(s,t)ds\\
	=&X_i^{\pi}(t_k)\mathcal{E}_i^{\pi}(t_k,t)+\int_{t_k}^tb_i\left(t,X^{\pi}(s)\right)\diamond\mathcal{E}_i^{\pi}(s,t)ds.
\end{align*}
This corresponds to (\ref{solution2}) and the proof is complete.

\section{Proof of Theorem \ref{main theorem 2}}\label{main theorem 2 proof}

Let $\varphi\in C^2_0([0,T]\times\mathbb R^d)$; then,
\begin{align*}
0=&\varphi(T,X^{\pi}(T))-\varphi(0,c)\\
=&\sum_{k=1}^N\varphi(t_k,X^{\pi}(t_k))-\varphi(t_{k-1},X^{\pi}(t_{k-1}))\\
=&\sum_{k=1}^N\int_{t_{k-1}}^{t_k}\left[\partial_t\varphi(t,X^{\pi}(t))+\sum_{i=1}^d\partial_{x_i}\varphi(t,X^{\pi}(t))\frac{d}{dt}X^{\pi}_i(t)\right]dt\\
=&\sum_{k=1}^N\int_{t_{k-1}}^{t_k}\partial_t\varphi(t,u(t,B(t_k)-B(t_{k-1});t_{k-1},X^{\pi}(t_{k-1})))dt\\
&+\sum_{k=1}^N\int_{t_{k-1}}^{t_k}\sum_{i=1}^d\partial_{x_i}\varphi(t,u(t,B(t_k)-B(t_{k-1});t_{k-1},X^{\pi}(t_{k-1})))\times\\
&\times\partial_tu_i(t,B(t_k)-B(t_{k-1});t_{k-1},X^{\pi}(t_{k-1}))dt.
\end{align*}
To ease the notation, we now suppress the explicit dependence on the initial conditions in the function $u$ and set $Z(k):=B(t_k)-B(t_{k-1})$; therefore, the previous identity reads
\begin{align}\label{z}
0=&\sum_{k=1}^N\int_{t_{k-1}}^{t_k}\partial_t\varphi(t,u(t,Z(k)))dt+\sum_{k=1}^N\sum_{i=1}^d\int_{t_{k-1}}^{t_k}\partial_{x_i}\varphi(t,u(t,Z(k)))\partial_tu_i(t,Z(k))dt.
\end{align}
We recall that $u_i$ is a classical solution of (\ref{PDE}) and hence we get
\begin{align*}
\partial_tu_i(t,Z(k))=-\sigma_i(t)\partial_{x_i} u_i(t,Z(k))+\sigma_i(t)\frac{Z_i(k)}{h}u_i(t,Z(k))+b_i(t,u(t,Z(k))). 
\end{align*}
Substituting this identity into (\ref{z}) yields
\begin{align*}
0=&\sum_{k=1}^N\int_{t_{k-1}}^{t_k}\partial_t\varphi(t,u(t,Z(k)))dt\nonumber\\
&-\sum_{k=1}^N\sum_{i=1}^d\int_{t_{k-1}}^{t_k}\partial_{x_i}\varphi(t,u(t,Z(k)))\sigma_i(t)\partial_{x_i} u_i(t,Z(k))dt\\
&+\sum_{k=1}^N\sum_{i=1}^d\int_{t_{k-1}}^{t_k}\partial_{x_i}\varphi(t,u(t,Z(k)))\sigma_i(t)\frac{Z_i(k)}{h}u_i(t,Z(k))dt\\
&+\sum_{k=1}^N\sum_{i=1}^d\int_{t_{k-1}}^{t_k}\partial_{x_i}\varphi(t,u(t,Z(k)))b_i(t,u(t,Z(k)))dt\\
=&\mathcal{A}-\mathcal{B}+\mathcal{C}+\mathcal{D},
\end{align*}
where
\begin{align*}
\mathcal{A}&:=\sum_{k=1}^N\int_{t_{k-1}}^{t_k}\partial_t\varphi(t,u(t,Z(k)))dt,\\
\mathcal{B}&:=\sum_{k=1}^N\sum_{i=1}^d\int_{t_{k-1}}^{t_k}\partial_{x_i}\varphi(t,u(t,Z(k)))\sigma_i(t)\partial_{x_i} u_i(t,Z(k))dt,\\
\mathcal{C}&:=\sum_{k=1}^N\sum_{i=1}^d\int_{t_{k-1}}^{t_k}\partial_{x_i}\varphi(t,u(t,Z(k)))\sigma_i(t)\frac{Z_i(k)}{h}u_i(t,Z(k))dt,\\
\mathcal{D}&:=\sum_{k=1}^N\sum_{i=1}^d\int_{t_{k-1}}^{t_k}\partial_{x_i}\varphi(t,u(t,Z(k)))b_i(t,u(t,Z(k)))dt.
\end{align*}
We now take the expectation of the first and last members above and get
\begin{align}\label{zz}
0=\mathbb{E}[\mathcal{A}]-\mathbb{E}[\mathcal{B}]+\mathbb{E}[\mathcal{C}]+\mathbb{E}[\mathcal{D}].
\end{align}
Let us analyse $\mathbb{E}[\mathcal{C}]$: 
\begin{align}\label{x}
\mathbb{E}[\mathcal{C}]=&\sum_{k=1}^N\sum_{i=1}^d\int_{t_{k-1}}^{t_k}\mathbb{E}\left[\partial_{x_i}\varphi(t,u(t,Z(k)))\sigma_i(t)\frac{Z_i(k)}{h}u_i(t,Z(k))\right]dt\nonumber\\
=&\sum_{k=1}^N\sum_{i=1}^d\int_{t_{k-1}}^{t_k}\mathbb{E}\left[\mathbb{E}\left[\partial_{x_i}\varphi(t,u(t,Z(k)))\sigma_i(t)\frac{Z_i(k)}{h}u_i(t,Z(k))\Big|\mathcal{F}_{t_{k-1}}\right]\right]dt;
\end{align}
here $\{\mathcal{F}_t\}_{t\in [0,T]}$ stands for the natural filtration of the Brownian motion $\{B(t)\}_{t\in [0,T]}$. We remark that $u(t,Z(k))$ depends implicitly also on the increments $Z(1),....,Z(k-1)$ through the initial condition; however, these increments are measurable with respect to the sigma-algebra $\mathcal{F}_{t_{k-1}}$. Therefore, the conditional expectation can be computed as follows
\begin{align*}
&\mathbb{E}\left[\partial_{x_i}\varphi(t,u(t,Z(k)))\sigma_i(t)\frac{Z_i(k)}{h}u_i(t,Z(k))\Big|\mathcal{F}_{t_{k-1}}\right]\\
&\quad=\int_{\mathbb{R}^d}\partial_{x_i}\varphi(t,u(t,x))\sigma_i(t)\frac{x_i}{h}u_i(t,x)\frac{e^{-|x|^2/2h}}{(2\pi h)^{d/2}}dx\\
&\quad=-\int_{\mathbb{R}^d}\partial_{x_i}\varphi(t,u(t,x))\sigma_i(t)u_i(t,x)\partial_{x_i}\left(\frac{e^{-|x|^2/2h}}{(2\pi h)^{d/2}}\right)dx\\
&\quad=\int_{\mathbb{R}^d}\sigma_i(t)\partial_{x_i}\left(\partial_{x_i}\varphi(t,u(t,x))u_i(t,x)\right)\frac{e^{-|x|^2/2h}}{(2\pi h)^{d/2}}dx\\
&\quad=\sum_{j=1}^d\int_{\mathbb{R}^d}\sigma_i(t)\partial_{x_j}\partial_{x_i}\varphi(t,u(t,x))\partial_{x_i}u_j(t,x)u_i(t,x)\frac{e^{-|x|^2/2h}}{(2\pi h)^{d/2}}dx\\
&\quad\quad+\int_{\mathbb{R}^d}\sigma_i(t)\partial_{x_i}\varphi(t,u(t,x))\partial_{x_i}u_i(t,x)\frac{e^{-|x|^2/2h}}{(2\pi h)^{d/2}}dx\\
&\quad=\mathbb{E}\left[\sum_{j=1}^d\sigma_i(t)\partial_{x_j}\partial_{x_i}\varphi(t,u(t,Z(k)))\partial_{x_i}u_j(t,Z(k))u_i(t,Z(k))\Big|\mathcal{F}_{t_{k-1}}\right]\\
&\quad\quad+\mathbb{E}\left[\sigma_i(t)\partial_{x_i}\varphi(t,u(t,Z(k)))\partial_{x_i}u_i(t,Z(k))\Big|\mathcal{F}_{t_{k-1}}\right];
\end{align*}
in the third equality we performed an integration by parts. Inserting the last expression in (\ref{x}) gives
\begin{align*}
\mathbb{E}[\mathcal{C}]=&\sum_{k=1}^N\sum_{i=1}^d\int_{t_{k-1}}^{t_k}\mathbb{E}\left[\mathbb{E}\left[\partial_{x_i}\varphi(t,u(t,Z(k)))\sigma_i(t)\frac{Z_i(k)}{h}u_i(t,Z(k))\Big|\mathcal{F}_{t_{k-1}}\right]\right]dt\\
=&\sum_{k=1}^N\sum_{i=1}^d\int_{t_{k-1}}^{t_k}\mathbb{E}\left[\sum_{j=1}^d\sigma_i(t)\partial_{x_j}\partial_{x_i}\varphi(t,u(t,Z(k)))\partial_{x_i}u_j(t,Z(k))u_i(t,Z(k))\right]dt\\
&+\sum_{k=1}^N\sum_{i=1}^d\int_{t_{k-1}}^{t_k}\mathbb{E}\left[\sigma_i(t)\partial_{x_i}\varphi(t,u(t,Z(k)))\partial_{x_i}u_i(t,Z(k))\right]dt.
\end{align*}
Note that last term above coincides with $\mathbb{E}[\mathcal{B}]$ which appear with a negative sign in (\ref{zz}); hence,
\begin{align*}
&-\mathbb{E}[\mathcal{B}]+\mathbb{E}[\mathcal{C}]\\
&\quad=\sum_{k=1}^N\sum_{i,j=1}^d\int_{t_{k-1}}^{t_k}\mathbb{E}\left[\sigma_i(t)\partial_{x_j}\partial_{x_i}\varphi(t,u(t,Z(k)))\partial_{x_i}u_j(t,Z(k))u_i(t,Z(k))\right]dt.
\end{align*}
Before recollecting all the parts of our computation, we make a further step; if we denote by $\mathcal{G}^{(k)}_{i,t}$ the sigma algebra generated by the random variable $u_i(t,Z(k))$, for $t\in [t_{k-1},t_{k}]$, $k=1,...,N$ and $i=1,...,d$, we can rewrite the expectation inside the integral above as
\begin{align*}
&\mathbb{E}\left[\sigma_i(t)\partial_{x_j}\partial_{x_i}\varphi(t,u(t,Z(k)))\partial_{x_i}u_j(t,Z(k))u_i(t,Z(k))\right]\\
&\quad=\mathbb{E}\left[\mathbb{E}\left[\sigma_i(t)\partial_{x_j}\partial_{x_i}\varphi(t,u(t,Z(k)))\partial_{x_i}u_j(t,Z(k))u_i(t,Z(k))\Big|\mathcal{G}^{(k)}_{i,t}\right]\right]\\
&\quad=\mathbb{E}\left[\sigma_i(t)\partial_{x_j}\partial_{x_i}\varphi(t,u(t,Z(k)))u_i(t,Z(k))\mathbb{E}\left[\partial_{x_i}u_j(t,Z(k))\Big|\mathcal{G}^{(k)}_{i,t}\right]\right]\\
&\quad=\mathbb{E}\left[\sigma_i(t)\partial_{x_j}\partial_{x_i}\varphi(t,u(t,Z(k)))u_i(t,Z(k))g^{(k)}_{ij}(t,u_i(t,Z(k)))\right],
\end{align*}  
where $g^{(k)}_{ij}:[t_{k-1},t_{k}]\times\mathbb{R}\to\mathbb{R}$ is a measurable function, whose existence is guaranteed by Doob's Lemma, chosen to satisfy
\begin{align}\label{g1}
g^{(k)}_{ij}(t,u_i(t,Z(k)))=\mathbb{E}\left[\partial_{x_i}u_j(t,Z(k))\Big|\mathcal{G}^{(k)}_{i,t}\right].
\end{align}
Now, starting from (\ref{zz}) and using the last two identities we obtain
\begin{align*}
0=&\mathbb{E}[\mathcal{A}]-\mathbb{E}[\mathcal{B}]+\mathbb{E}[\mathcal{C}]+\mathbb{E}[\mathcal{D}]\\
=&\sum_{k=1}^N\int_{t_{k-1}}^{t_k}\mathbb{E}[\partial_t\varphi(t,u(t,Z(k)))]dt\\
&+\sum_{k=1}^N\sum_{i,j=1}^d\int_{t_{k-1}}^{t_k}\mathbb{E}\left[\sigma_i(t)\partial_{x_j}\partial_{x_i}\varphi(t,u(t,Z(k)))u_i(t,Z(k))g^{(k)}_{ij}(t,u_i(t,Z(k)))\right]dt\\
&+\sum_{k=1}^N\sum_{i=1}^d\int_{t_{k-1}}^{t_k}\mathbb{E}[\partial_{x_i}\varphi(t,u(t,Z(k)))b_i(t,u(t,Z(k)))]dt\\
=&\int_0^T\mathbb{E}[\partial_t\varphi(t,X^{\pi}(t))]dt\\
&+\sum_{k=1}^N\sum_{i,j=1}^d\int_{t_{k-1}}^{t_k}\mathbb{E}\left[\sigma_i(t)\partial_{x_j}\partial_{x_i}\varphi(t,X^{\pi}(t))X_i^{\pi}(t)g^{(k)}_{ij}(t,X_i^{\pi}(t))\right]dt\\
&+\int_0^T\sum_{i=1}^d\mathbb{E}[\partial_{x_i}\varphi(t,X^{\pi}(t))b_i(t,X^{\pi}(t))]dt\\
=&\int_0^T\mathbb{E}[\partial_t\varphi(t,X^{\pi}(t))]dt\\
&+\sum_{i,j=1}^d\int_0^T\mathbb{E}\left[\sigma_i(t)\partial_{x_j}\partial_{x_i}\varphi(t,X^{\pi}(t))X_i^{\pi}(t)g_{ij}(t,X_i^{\pi}(t))\right]dt\\
&+\int_0^T\sum_{i=1}^d\mathbb{E}[\partial_{x_i}\varphi(t,X^{\pi}(t))b_i(t,X^{\pi}(t))]dt,
\end{align*}
where $g_{ij}:[0,T]\times\mathbb{R}\to\mathbb{R}$ is defined by
\begin{align}\label{g2}
g_{ij}(t,y):=g^{(k)}_{ij}(t,y),\quad\mbox{ if } t\in [t_{k-1},t_k].
\end{align}
Observe that the last member above contains expectations of functions of the random vector $X^{\pi}(t)$, for $t\in [0,T]$; therefore, writing the law of this random vector as 
\begin{align*}
\mu^{\pi}(t,A):=\mathbb{P}(X^{\pi}(t)\in A),\quad A\in\mathcal{B}(\mathbb{R}^d),
\end{align*} 
we can write 
\begin{align*}
0=&\int_0^T\int_{\mathbb{R}^d}\partial_t\varphi(t,x)d\mu^{\pi}(t,x)dt+\sum_{i,j=1}^d\int_0^T\int_{\mathbb{R}^d}\sigma_i(t)\partial_{x_j}\partial_{x_i}\varphi(t,x)x_ig_{ij}(t,x_i)d\mu^{\pi}(t,x)dt\\
&+\sum_{i=1}^d\int_0^T\int_{\mathbb{R}^d}\partial_{x_i}\varphi(t,x)b_i(t,x)d\mu^{\pi}(t,x)dt\\
=&\int_0^T\int_{\mathbb{R}^d}\left[\partial_t\varphi(t,x)+\sum_{i,j=1}^d\sigma_i(t)\partial^2_{x_ix_j}\varphi(t,x)x_ig_{ij}(t,x_i)+\langle b(t,x),\nabla\varphi(t,x)\rangle \right]d\mu^{\pi}(t,x)dt.
\end{align*}
The last equalities hold for any test function $\varphi\in C^2_0([0,T]\times\mathbb R^d)$ and this completes the proof of Theorem \ref{main theorem 2}. 

\section{Proof of Theorem \ref{main theorem 3}}\label{main theorem 3 proof}

The aim of this section is to prove that the mild solution of
\begin{align}\label{WZ bis}
\begin{cases}
\frac{dX_i^{\pi}(t)}{dt}=b_i(t,X^{\pi}(t))+\sigma_i(t)X^{\pi}_i(t)\diamond \frac{dB_i^{\pi}(t)}{dt},\\
\quad\quad\mbox{for }t\in]0,T]\mbox{ and } i=1,...,d;\\
X^{\pi}_i(0)=c_i\in\mathbb{R},\quad\mbox{for } i=1,...,d,
\end{cases}
\end{align}
as defined in (\ref{def X_t}), converges in $\mathbb{L}^1(\Omega)$ to the unique strong solution of the It\^o SDE
\begin{align}\label{ItoSDE bis}
\begin{cases}
dX_i(t)=b_i(t,X(t))dt+\sigma_i(t)X_i(t)dB_i(t),\\
\quad\quad\mbox{for }t\in]0,T]\mbox{ and } i=1,...,d;\\
X_i(0)=c_i\in\mathbb{R},\quad\mbox{for } i=1,...,d.
\end{cases}
\end{align}
First of all, by means of the It\^o formula we rewrite equation (\ref{ItoSDE bis}) in a form that resembles identity (\ref{solution}). In fact, setting
\begin{align*}
\mathtt{E}_i(s,t):=e^{-\int_s^t\sigma_i(r)dB_i(r)-\frac{1}{2}\int_s^t\sigma_i(r)^2dr},\quad 0\leq s\leq t\leq T,
\end{align*}
and
\begin{align*}
\mathcal{E}_i(s,t):=e^{\int_s^t\sigma_i(r)dB_i(r)-\frac{1}{2}\int_s^t\sigma_i(r)^2dr},\quad 0\leq s\leq t\leq T,
\end{align*}
we write
\begin{align*}
d\left(X_i(t)\diamond \mathtt{E}_i(0,t)\right)=&d(\mathtt{T}_{\sigma_i}X_i(t)\cdot\mathtt{E}_i(0,t))\\
=&\mathtt{E}_i(0,t)\cdot d\mathtt{T}_{\sigma_i}X_i(t)+\mathtt{T}_{\sigma_i}X_i(t)\cdot d\mathtt{E}_i(0,t)\\
&+d\mathtt{T}_{\sigma_i}X_i(t)\cdot d\mathtt{E}_i(0,t).
\end{align*}
Now,
\begin{align*}
d\mathtt{T}_{\sigma_i}X_i(t)&=[b_i(t,\mathtt{T}_{\sigma_i}X(t))+\sigma_i(t)^2\mathtt{T}_{\sigma_i}X_i(t)]dt+\sigma_i(t)\mathtt{T}_{\sigma_i}X_i(t)dB_i(t),\\
d\mathtt{E}_i(0,t)&=-\sigma_i(t)\mathtt{E}_i(0,t)dB_i(t),
\end{align*}
and hence
\begin{align*}
&d\left(X_i(t)\diamond \mathtt{E}_i(0,t)\right)\\
&\quad=[b_i(t,\mathtt{T}_{\sigma_i}X(t))\mathtt{E}_i(0,t)+\sigma_i(t)^2\mathtt{T}_{\sigma_i}X_i(t)\mathtt{E}_i(0,t)]dt+\sigma_i(t)\mathtt{T}_{\sigma_i}X_i(t)\mathtt{E}_i(0,t)dB_i(t)\\
&\quad\quad-\sigma_i(t)\mathtt{T}_{\sigma_i}X_i(t)\mathtt{E}_i(0,t)dB_i(t)-\sigma_i(t)^2\mathtt{T}_{\sigma_i}X_i(t)\mathtt{E}_i(0,t)dt\\
&=b_i(t,\mathtt{T}_{\sigma_i}X(t))\mathtt{E}_i(0,t)dt\\
&=b_i(t,X(t))\diamond\mathtt{E}_i(0,t)dt.
\end{align*}
This is equivalent to
\begin{align*}
X_i(t)\diamond \mathtt{E}_i(0,t)=c_i+\int_0^tb_i(s,X(s))\diamond\mathtt{E}_i(0,s)ds,
\end{align*}
or
\begin{align*}
X_i(t)&=c_i\mathcal{E}_i(0,t)+\int_0^tb_i(s,X(s))\diamond\mathtt{E}_i(0,s)\diamond\mathcal{E}_i(0,t)ds\\
&=c_i\mathcal{E}_i(0,t)+\int_0^tb_i(s,X(s))\diamond\mathcal{E}_i(s,t)ds.
\end{align*}
Here, we utilized the equality
\begin{align*}
\mathtt{E}_i(0,t)\diamond\mathcal{E}_i(0,t)=1,\quad\mbox{ for all $t\in [0,T]$}.
\end{align*}
Therefore, the solution of the It\^o SDE (\ref{ItoSDE bis}) verifies the integral identity
\begin{align}\label{mild Ito}
X_i(t)=c_i\mathcal{E}_i(0,t)+\int_0^tb_i(s,X(s))\diamond\mathcal{E}_i(s,t)ds,
\end{align}
for all $t\in [0,T]$ and $i=1,...,d$. We are now ready to prove the convergence:
\begin{align*}
&|X_i^{\pi}(t)-X_i(t)|&\\
&\quad=\left|c_i\left(\mathcal E_i^{\pi}(t,0)-\mathcal E_i(0,t)\right)+\int_{0}^{t}b_i(s,X^{\pi}(s))\diamond \mathcal E_i^{\pi}(s,t)-b_i\left(s,X(s)\right)\diamond \mathcal E_i(s,t)ds\right|\\
&\quad\leq|c_i|\left|\mathcal E_i^{\pi}(0,t)-\mathcal E_i(0,t)\right|+\int_{0}^{t}\left|b_i(s,X^{\pi}(s))\diamond \mathcal E_i^{\pi}(s,t)-b_i\left(s,X(s)\right)\diamond \mathcal E_i(s,t)\right|ds\\
&\quad\leq|c_i|\left|\mathcal E_i^{\pi}(0,t)-\mathcal E_i(0,t)\right|+\int_{0}^{t}\left|b_i(s,X^{\pi}(s))\diamond \mathcal E_i^{\pi}(s,t)-b_i(s,X(s))\diamond\mathcal E_i^{\pi}(s,t)\right|ds\\
&\quad\quad+\int_0^t\left|b_i(s,X(s))\diamond\mathcal E_i^{\pi}(s,t)-b_i\left(s,X(s)\right)\diamond \mathcal E_i(s,t)\right|ds\\
&\quad\leq|c_i|\left|\mathcal E_i^{\pi}(0,t)-\mathcal E_i(0,t)\right|+\int_{0}^{t}\left|b_i(s,X^{\pi}(s))-b_i(s,X(s))\right|\diamond\mathcal E_i^{\pi}(s,t)ds\\
&\quad\quad+\int_0^t\left|b_i(s,X(s))\diamond\mathcal E_i^{\pi}(s,t)-b_i\left(s,X(s)\right)\diamond \mathcal E_i(s,t)\right|ds\\
&\quad\leq|c_i|\left|\mathcal E_i^{\pi}(0,t)-\mathcal E_i(0,t)\right|+L\int_{0}^{t}\sum_{j=1}^d\left|X_j^{\pi}(s)-X_j(s)\right|\diamond\mathcal E_i^{\pi}(s,t)ds\\
&\quad\quad+\int_0^t\left|b_i(s,X(s))\diamond\left(\mathcal E_i^{\pi}(s,t)-\mathcal E_i(s,t)\right)\right|ds;
\end{align*}
In the last two estimates we utilized inequality (\ref{wick absolute value}) together with the Lipschitz continuity of $b$, which is implied by Assumption \ref{assumptions}. We now take the expectation of the first and last members above to get
\begin{align*}
&\mathbb{E}[|X_i^{\pi}(t)-X_i(t)|]\\
&\quad\leq |c_i|\mathbb{E}\left[\left|\mathcal E_i^{\pi}(0,t)-\mathcal E_i(0,t)\right|\right]+L\int_{0}^{t}\sum_{j=1}^d\mathbb{E}\left[\left|X_j^{\pi}(s)-X_j(s)\right|\right]ds\\
&\quad\quad+\int_0^t\mathbb{E}\left[\left|b_i(s,X(s))\diamond\left(\mathcal E_i^{\pi}(s,t)-\mathcal E_i(s,t)\right)\right|\right]ds.
\end{align*}
The previous inequality is valid for all $i=1,...,d$ and $t\in [0,T]$; therefore, summing over $i$ and setting
\begin{align*}
\mathtt{X}^{\pi}(t):=\sum_{i=1}^d\mathbb{E}\left[\left|X_i^{\pi}(s)-X_i(s)\right|\right],
\end{align*}
we obtain
\begin{align*}
\mathtt{X}^{\pi}(t)\leq&\sum_{i=1}^d|c_i|\mathbb{E}\left[\left|\mathcal E_i^{\pi}(0,t)-\mathcal E_i(0,t)\right|\right]+Ld\int_{0}^{t}\mathtt{X}^{\pi}(s)ds\\
&+\sum_{i=1}^d\int_0^t\mathbb{E}\left[\left|b_i(s,X(s))\diamond\left(\mathcal E_i^{\pi}(s,t)-\mathcal E_i(s,t)\right)\right|\right]ds\\
=&\mathcal{M}^{\pi}(t)+Ld\int_{0}^{t}\mathtt{X}^{\pi}(s)ds,
\end{align*}
with
\begin{align*}
\mathcal{M}^{\pi}(t):=&\sum_{i=1}^d|c_i|\mathbb{E}\left[\left|\mathcal E_i^{\pi}(0,t)-\mathcal E_i(0,t)\right|\right]\\
&+\sum_{i=1}^d\int_0^t\mathbb{E}\left[\left|b_i(s,X(s))\diamond\left(\mathcal E_i^{\pi}(s,t)-\mathcal E_i(s,t)\right)\right|\right]ds.
\end{align*}
According to Gronwall's inequality the previous estimate yields
\begin{align}\label{gronwall}
\mathtt{X}^{\pi}(t)\leq\mathcal{M}^{\pi}(t)+L d \int_0^t \mathcal{M}^{\pi}(s)e^{L d (t-s)}ds;
\end{align}
the proof will be complete if we show that $\mathcal{M}^{\pi}(t)$ is bounded for all $t\in [0,T]$ and any finite partition $\pi$ and that
\begin{align*}
\lim_{\Vert\pi\Vert\to 0}\mathcal{M}^{\pi}(t)=0,\quad\mbox{ for all $t\in [0,T]$};
\end{align*}
this will allow us to use dominated convergence for the Lebesgue integral appearing in (\ref{gronwall}) and conclude that
\begin{align*}
\lim_{\Vert\pi\Vert\to 0}\mathtt{X}^{\pi}(t)=\lim_{\Vert\pi\Vert\to 0}\sum_{i=1}^d\mathbb{E}\left[\left|X_i^{\pi}(s)-X_i(s)\right|\right]=0.
\end{align*}
We start with the boundedness:
\begin{align*}
\mathcal{M}^{\pi}(t)\leq&\sum_{i=1}^d|c_i|\left(\mathbb{E}\left[\left|\mathcal E_i^{\pi}(0,t)\right|\right]+\mathbb{E}\left[\left|\mathcal E_i(0,t)\right|\right]\right)\\
&+\sum_{i=1}^d\int_0^t\mathbb{E}\left[\left|b_i(s,X(s))\diamond\mathcal E_i^{\pi}(s,t)-b_i(s,X(s))\diamond\mathcal E_i(s,t)\right|\right]ds\\
\leq& 2\sum_{i=1}^d|c_i|+\sum_{i=1}^d\int_0^t\mathbb{E}\left[\left|b_i(s,X(s))\diamond\mathcal E_i^{\pi}(s,t)\right|\right]+\mathbb{E}\left[\left|b_i(s,X(s))\diamond\mathcal E_i(s,t)\right|\right]ds\\
\leq& 2\sum_{i=1}^d|c_i|+\sum_{i=1}^d\int_0^t\mathbb{E}\left[|b_i(s,X(s))|\diamond\mathcal E_i^{\pi}(s,t)\right]+\mathbb{E}\left[|b_i(s,X(s))|\diamond\mathcal E_i(s,t)\right]ds\\
\leq& 2\sum_{i=1}^d|c_i|+2dMt.
\end{align*}	
We now check the convergence:
\begin{align*}
\lim_{\Vert\pi\Vert\to 0}\mathcal{M}^{\pi}(t)=&\lim_{\Vert\pi\Vert\to 0}\sum_{i=1}^d|c_i|\mathbb{E}\left[\left|\mathcal E_i^{\pi}(0,t)-\mathcal E_i(0,t)\right|\right]\\
&+\lim_{\Vert\pi\Vert\to 0}\sum_{i=1}^d\int_0^t\mathbb{E}\left[\left|b_i(s,X(s))\diamond\left(\mathcal E_i^{\pi}(s,t)-\mathcal E_i(s,t)\right)\right|\right]ds\\
=&\sum_{i=1}^d\lim_{\Vert\pi\Vert\to 0}\int_0^t\mathbb{E}\left[\left|b_i(s,X(s))\diamond\left(\mathcal E_i^{\pi}(s,t)-\mathcal E_i(s,t)\right)\right|\right]ds.
\end{align*}
We now prove that we can take the last limit inside the integral; first of all, note that the integrand is bounded: in fact,
\begin{align*}
\mathbb{E}\left[\left|b_i(s,X(s))\diamond\left(\mathcal E_i^{\pi}(s,t)-\mathcal E_i(s,t)\right)\right|\right]=&\mathbb{E}\left[\left|b_i(s,X(s))\diamond\mathcal E_i^{\pi}(s,t)-b_i(s,X(s))\diamond\mathcal E_i(s,t)\right|\right]\\
\leq&\mathbb{E}\left[\left|b_i(s,X(s))\diamond\mathcal E_i^{\pi}(s,t)\right|\right]+\mathbb{E}\left[\left|b_i(s,X(s))\diamond\mathcal E_i(s,t)\right|\right]\\
\leq&\mathbb{E}\left[|b_i(s,X(s))|\diamond\mathcal E_i^{\pi}(s,t)\right]+\mathbb{E}\left[|b_i(s,X(s))|\diamond\mathcal E_i(s,t)\right]\\
=&\mathbb{E}\left[|b_i(s,X(s))|\right]+\mathbb{E}\left[|b_i(s,X(s))|\right]\\
\leq&2M.
\end{align*}
We proceed by proving that 
\begin{align*}
\lim_{\Vert\pi\Vert\to 0}\mathbb{E}\left[\left|b_i(s,X(s))\diamond\left(\mathcal E_i^{\pi}(s,t)-\mathcal E_i(s,t)\right)\right|\right]=0.
\end{align*} 
Let us rewrite the expected value as follows:
\begin{align*}
&\mathbb{E}\left[\left|b_i(s,X(s))\diamond\left(\mathcal E_i^{\pi}(s,t)-\mathcal E_i(s,t)\right)\right|\right]\\
&\quad=\mathbb{E}\left[\left|b_i(s,X(s))\diamond\mathcal E_i^{\pi}(s,t)-b_i(s,X(s))\diamond\mathcal E_i(s,t)\right|\right]\\
&\quad=\mathbb{E}\left[\left|\mathtt{T}_{\sigma_i,\pi}b_i(s,X(s))\mathcal E_i^{\pi}(s,t)-\mathtt{T}_{\sigma_i}b_i(s,X(s))\mathcal E_i(s,t)\right|\right]\\
&\quad\leq\mathbb{E}\left[\left|\mathtt{T}_{\sigma_i,\pi}b_i(s,X(s))\mathcal E_i^{\pi}(s,t)-\mathtt{T}_{\sigma_i}b_i(s,X(s))\mathcal E_i^{\pi}(s,t)\right|\right]\\
&\quad\quad+\mathbb{E}\left[\left|\mathtt{T}_{\sigma_i}b_i(s,X(s))\mathcal E_i^{\pi}(s,t)
-\mathtt{T}_{\sigma_i}b_i(s,X(s))\mathcal E_i(s,t)\right|\right]\\
&\quad=\mathbb{E}\left[|b_i(s,\mathtt{T}_{\sigma_i,\pi}X(s))-b_i(s,\mathtt{T}_{\sigma_i}X(s))|\mathcal E_i^{\pi}(s,t)\right]\\
&\quad\quad+\mathbb{E}\left[|b_i(s,\mathtt{T}_{\sigma_i}X(s))||\mathcal E_i^{\pi}(s,t)
-\mathcal E_i(s,t)|\right]\\
&\quad\leq L\mathbb{E}\left[|\mathtt{T}_{\sigma_i,\pi}X(s)-\mathtt{T}_{\sigma_i}X(s)|\mathcal E_i^{\pi}(s,t)\right]\\
&\quad\quad+M\mathbb{E}\left[|\mathcal E_i^{\pi}(s,t)
-\mathcal E_i(s,t)|\right].
\end{align*} 
Hence,
\begin{align*}
&\lim_{\Vert\pi\Vert\to 0}\mathbb{E}\left[\left|b_i(s,X(s))\diamond\left(\mathcal E_i^{\pi}(s,t)-\mathcal E_i(s,t)\right)\right|\right]\\
&\quad\leq\lim_{\Vert\pi\Vert\to 0}L\mathbb{E}\left[|\mathtt{T}_{\sigma_i,\pi}X(s)-\mathtt{T}_{\sigma_i}X(s)|\mathcal E_i^{\pi}(s,t)\right]\\
&\quad\quad+\lim_{\Vert\pi\Vert\to 0}M\mathbb{E}\left[|\mathcal E_i^{\pi}(s,t)
-\mathcal E_i(s,t)|\right].
\end{align*} 
By the properties of the translation operator,
\begin{align*}
\lim_{\Vert\pi\Vert\to 0}\mathtt{T}_{\sigma_i,\pi}X(s)=\mathtt{T}_{\sigma_i}X(s),\quad\mbox{ in $\mathbb{L}^p(\Omega)$ for all $p\geq 1$};
\end{align*}
on the other hand
\begin{align*}
\lim_{\Vert\pi\Vert\to 0}\mathcal E_i^{\pi}(s,t)=\mathcal E_i(s,t),\quad\mbox{ in $\mathbb{L}^p(\Omega)$ for all $p\geq 1$}.
\end{align*}
These two facts imply
\begin{align*}
\lim_{\Vert\pi\Vert\to 0}L\mathbb{E}\left[|\mathtt{T}_{\sigma_i,\pi}X(s)-\mathtt{T}_{\sigma_i}X(s)|\mathcal E_i^{\pi}(s,t)\right]=0,
\end{align*}
completing the proof.

\end{document}